\documentclass[a4paper,12pt]{amsart}

\usepackage[english]{babel}
\usepackage{amsthm, latexsym, gastex, amssymb,url,listings}
\usepackage[mathcal]{eucal}
\usepackage{graphicx,url}

\copyrightinfo{2011}{Simone Ugolini}

\DeclareMathOperator{\ord}{ord}

\renewcommand{\phi}[0]{\varphi}
\renewcommand{\theta}[0]{\vartheta}
\renewcommand{\epsilon}[0]{\varepsilon}

\newcommand{\Z}{\text{$\mathbf{Z}$}}

\newcommand{\Pro}{\text{$\mathbf{P}^1$}}
\newcommand{\F}{\text{$\mathbf{F}$}}

\newtheorem{theorem}{Theorem}[section]
\newtheorem{lemma}[theorem]{Lemma}

\theoremstyle{definition}

\theoremstyle{remark}

\numberwithin{equation}{section}

\begin{document}

\bibliographystyle{amsalpha}

\date{}

\title[GS]
{Graphs associated with the map $X \mapsto X + X^{-1}$
in finite fields of characteristic three}

\author{S.~Ugolini}
\email{sugolini@gmail.com} 

\begin{abstract}
In \cite{SU2} we described the structure of the graphs associated with the iterations of the map $x \mapsto x+x^{-1}$ over finite fields of characteristic two. In this paper we extend our study to finite fields of characteristics three. 
\end{abstract}

\maketitle
\section{Introduction}
Let $\F_q$ be a finite field with $q$ elements for some positive integer $q$. We can define a map $\theta$ on $\Pro (\F_q) = \F_q \cup \{\infty \}$ in such a way:
\begin{displaymath}
\theta(x) =   
\begin{cases}
x+x^{-1} & \text{if $x \not =0, \infty$}\\
\infty & \text{if $x = 0$ or $\infty$}
\end{cases}
\end{displaymath}

We associate a graph with the map $\theta$ over $\F_{q}$, labelling the vertices of the graph by the elements of $\Pro (\F_{q})$. Moreover, if $\alpha \in \Pro (\F_{q})$ and $\beta = \theta(\alpha)$, then a directed edge connects the vertex $\alpha$ with the vertex $\beta$.
If $\gamma \in \Pro (\F_q)$ and $\theta^k (\gamma) = \gamma$, for some positive integer $k$, then $\gamma$ belongs to a cycle of length $k$ or a divisor of $k$. The smallest among these integers $k$ is the period $l$ of $\gamma$ with respect to the map $\theta$ and the set $\{\theta^i (\gamma) : 0 \leq i < l \}$ is the cycle of length $l$ containing $\gamma$.
An element $\gamma$ belonging to a cycle can be the root of a reverse-directed tree, provided that $\gamma = \theta (\alpha)$, for some $\alpha$ which is not contained in any cycle. 

In \cite{SU2} we dealt with the characteristic $2$ case. There we noticed that the map $\theta$ is strictly related to the duplication map over Koblitz curves. Later we carried out some experiments in characteristics greater than $5$, but the resulting graphs seemed not to present notable symmetries.

In characteristics $3$ and $5$, in analogy with our previous work \cite{SU2}, the graphs exhibit remarkable symmetries.  In this paper we present the characteristic $3$ case.

In characteristic three the structure of the graphs can be described relying upon the fact that $\theta$ is conjugated to the inverse of the square mapping. Fixed a finite field $\F_{3^n}$ we provide the following information about the graph associated with $\theta$:
\begin{itemize}
\item the lengths, the number of the cycles and the number of the connected componenents (Theorem \ref{l2});
\item the depth and the properties of the trees (Theorem \ref{l3}).
\end{itemize}

\section{Structure of the graphs in characteristic three}
In characteristic three the iterations of the map $\theta$ can be studied relying upon the consideration that $\theta$ is conjugated to the inverse of the square map. 
Indeed, if $x$ is any element of a field of characteristic $3$, then 
\begin{equation}\label{conj}
\theta(x) = \psi \circ s \circ \psi(x),
\end{equation} 
where $s$ and $\psi$ are functions defined on $\Pro (\F_{3^n})$ as follows
\begin{displaymath}
s(x)=
\begin{cases}
x^{-2} & \text{if $x \in \F_{3^n}^*$}\\
0 & \text{if $x = \infty$}\\
\infty & \text{if $x = 0$}
\end{cases}
\quad
\psi(x) =
\begin{cases}
\dfrac{x+1}{x-1} & \text{if $x \in \Pro (\F_{3^n}) \backslash  \{1, \infty \}$}\\
1 & \text{if $x = \infty$}\\
\infty & \text{if $x =1$}
\end{cases}
\end{displaymath}

We note that $\psi$ is a self-inverse map over $\Pro (\F_{3^n})$, namely $\psi^2 (x) = x$ for any $x \in \Pro (\F_{3^n})$. Therefore the following holds for the $k$-th iterate of $\theta$:
\begin{equation*}
\theta^k (x) = \psi \circ s^{k} \circ \psi(x).
\end{equation*}

We say that an element $x \in \Pro(\F_{3^n})$ is $\theta$-periodic (resp. $s$-periodic) iff $\theta^k(x) = x$ (resp. $s^k (x) = x)$, for some positive integer $k$. The smallest such $k$ will be called the period of $x$ with respect to the map $\theta$ (resp. $s$).

We prove the following characterization of $\theta$-periodic points.
\begin{lemma}\label{l1}
Let $n$ be a positive integer.
\begin{itemize}
\item The elements $1$ and $-1$ are $\theta$-periodic of period $2$ and form a cycle of length 2.
\item The element $\infty$ is $\theta$-periodic of period $1$. 
\item An element $\alpha \in \Pro( \F_{3^n} ) \backslash \left\{-1 ,1, \infty  \right\}$ is $\theta$-periodic of period $k$  if and only if $\psi(\alpha)$ is $s$-periodic of period $k$. Moreover, the integer $k$ is odd and  is equal to the multiplicative order $\ord_d (-2)$ of $-2$ in $(\Z / d \Z)^*$, where $d$ is the multiplicative order of $\psi(\alpha)$ in $\F_{3^n}^*$.
\end{itemize}
\end{lemma}
\begin{proof}
An element $\alpha \in \Pro ( \F_{3^n} )$ is $\theta$-periodic if and only if there exists a positive integer $k$ such that
$\theta^k (\alpha) = \psi \circ s^k \circ \psi (\alpha) = \alpha$.
 Let $\beta = \psi(\alpha)$. We have the following equivalences:
\begin{eqnarray*}
\theta^k (\alpha) =  \alpha & \Leftrightarrow &  \psi \circ s^k \circ \psi (\alpha) = \alpha \Leftrightarrow s^k (\psi(\alpha)) = \psi(\alpha).
\end{eqnarray*}

In virtue of what we have just proved $\alpha$ is $\theta$-periodic of period $k$ if and only if $\beta = \psi(\alpha)$ is $s$-periodic of period $k$.

If $\alpha \in \Pro (\F_{3^n}) \backslash \{-1,1, \infty \}$, then $\psi(\alpha) \not \in \{0, \infty \}$. Therefore, $\alpha$ is $\theta$-periodic of period $k$ if and only if $\beta^{(-2)^k} = \beta$, namely $\beta^{(-2)^k-1} = 1$. This latter is true if and only if $d$ divides $(-2)^k-1$. That means that $k = \ord_{d} (-2)$.  Since $(-2)^k-1$ is an odd integer, this is possible iff the multiplicative order of $\beta$ in $\F_{3^n}^*$ is odd.

Finally, consider the elements $1, -1, \infty$. Since $\theta (-1) = 1$, $\theta (1) = -1$ and $\theta (\infty) = \infty$ the first two statements of the claim are proved.
\end{proof}

In the following Theorem the lengths and the number of cycles of the graph associated with $\theta$ over $\F_{3^n}$ are given.
\begin{theorem}\label{l2}
Let $n$ be a positive integer and $D = \left\{d_1, \dots, d_m \right\}$ the set of the distinct odd integers greater than $1$ which divide $3^n-1$. Denote by $ord_{d_i} (-2)$ the multiplicative order of $-2$ in $(\Z / d_i \Z)^*$. Consider the set
\begin{displaymath}
L = \left\{ \ord_{d_i} (-2) : 1 \leq i \leq m \right\} = \left\{l_1, \dots, l_r \right\}
\end{displaymath}
of cardinality $r$, where  $r \leq m$, and the map
\begin{eqnarray*}
l : D & \to & L\\
d_i & \mapsto & \ord_{d_i} (-2).
\end{eqnarray*}
Then:
\begin{itemize}
\item $L \cap \{1, 2 \} = \emptyset$;
\item the length of a cycle in the graph associated with $\theta$ over $\Pro (\F_{3^n})$ is a positive integer belonging to $L \cup \{1, 2 \}$;
\item there is one cycle of length $2$ formed by $1$ and $-1$ and one cycle of length $1$ formed by $\infty$;
\item for any $1 \leq k \leq r$ there are
\begin{equation*}
c_k = \dfrac{1}{l_k} \cdot \sum_{d_i \in l^{-1} (l_k)}  \phi(d_i)
\end{equation*}
cycles of length $l_k$;
\item the number of connected components of the graph is 
\begin{displaymath}
2 + \displaystyle\sum_{k = 1}^r c_k.
\end{displaymath}
\end{itemize}
\end{theorem}
\begin{proof}
Since any element of $L$ is equal to $\ord_d (-2)$ for some odd integer $d > 1$, then $1$ is not contained in $L$. Moreover $2$ is not contained in $L$ too. In fact, $\ord_d (-2) = 2$ if and only if $d = 3$. But this is not possible, because $3$ does not divide $3^n-1$.

In Lemma \ref{l1} we proved that $\pm 1$ are the only $\theta$-periodic elements of order $2$, while $\infty$ is $\theta$-periodic of period $1$. All $\theta$-periodic elements of $\Pro (\F_{3^n}) \backslash \{-1, 0, 1 \}$ have odd period $k$, where $k$ is the multiplicative order of $-2$ in $(\Z / d \Z)^*$, for some odd integer $d$ which divides $3^n-1$. Therefore, the length of a cycle is an integer belonging to $L \cup \{1, 2 \}$. 

Take an odd divisor $d_i > 1$ of $3^n-1$. In $\F_{3^n}^*$ there are $\phi(d_i)$ elements of order $d_i$. Since $\psi$ is a bijection on $\Pro (\F_{3^n})$, then each of these elements is of the form $\psi(\alpha)$ for some $\alpha \in \Pro (\F_{3^n})$. 

Consider an element $l_k \in L$. Since $\ord_{d_i} (-2) = l_k$ if and only if $d_i \in l^{-1} (l_k)$, then the number of cycles of length $l_k$ is given by $c_k$. Moreover, since any element of $\Pro (F_{3^n})$ is finally periodic, we conclude that the number of connected components of the graph is equal to the number of the cycles.
\end{proof}

We aim at describing the trees rooted at periodic elements of $\Pro (\F_{3^n})$. Before proceeding, we note that the elements $1$ and $-1$ are not roots of any tree. This can be easily seen, since $\theta(x) = \pm 1$ iff $x^2 \pm x +1 = 0$. These two equations are equivalent to the equations $(x-1)^2 = 0$ and $(x+1)^2=0$, since the underlying field has characteristic $3$, and the roots of the two equations are respectively $1$ and $-1$.

Before studying the structure of the trees, we prove the following preliminary result.
\begin{lemma}\label{l4}
Let $n$ be a positive integer and $2^e$, for some positive integer $e$, the greatest power of $2$ dividing $3^n-1$.
Let $\gamma \in \F_{3^n}$ be a non-$\theta$-periodic point (in particular $\gamma \not \in \{1, -1 \}$). Then, $\theta (x) = \gamma$ for exactly two distinct elements $x \in \F_{3^n}$, provided that $\ord(\psi(\gamma)) \not \equiv 0 \pmod{2^e}$, where $\ord(\psi(\gamma))$ is the multiplicative order of $\psi(\gamma)$ in $\F_{3^n}^*$. If, on the contrary, $ \ord ( \psi (\gamma)) \equiv 0 \pmod {2^e}$, then there is no $x \in \F_{3^n}$ such that $\theta (x) = \gamma$.
\end{lemma}
\begin{proof}
Take $\gamma$ as in the hypotheses. We note that, if $\theta (x) = \gamma$, then $x \not \in \{-1, 0, 1 \}$, since $\theta(-1) = 1, \theta(1) = -1$ and $\theta(0) = \infty$, but $\gamma \in \F_{3^n} \backslash \{1, -1 \}$. Hence, there exists $x \in \F_{3^n}$ such that $\theta(x) = \gamma$ iff $\psi \circ s \circ \psi (x) = \gamma$, namely iff $\psi(x)^{-2} = \psi(\gamma)$. This is equivalent to saying that $\psi(\gamma)$ is a quadratic residue in $\F_{3^n}$. This is true iff $\psi(\gamma)^{(3^n-1)/2} = 1$ in $\F_{3^n}^*$, namely iff $\ord(\psi(\gamma)) \mid \frac{3^n-1}{2}$. This latter is equivalent to saying that $\ord(\psi(\gamma)) \not \equiv 0 \pmod{2^e}$.
\end{proof}

In the following result the depth of the reversed binary trees rooted at $\theta$-periodic elements is given.
\begin{theorem}\label{l3}
Let $\alpha \in  \Pro (\F_{3^n}) \backslash \{1, -1 \} $ be a $\theta$-periodic point.
Then, $\alpha$ is the root of a reversed binary tree of depth $e$, where $2^e$ is the greatest power of 2 which divides $3^n-1$. In particular:
\begin{itemize}
\item there are $2^{k-1}$ vertices at any level $1 \leq k \leq e$;
\item the root has one child and all the other vertices at any level $k < e$ have two children;
\item if $\beta \in \F_{3^n}$ belongs to the level $k > 0$ of the tree rooted at $\alpha$, then $2^k$ is the greatest power of $2$ dividing $\ord(\psi(\beta))$.
\end{itemize}
\end{theorem}

\begin{proof}
If $\alpha = \infty$, then $\alpha$ is $\theta$-periodic of period $1$. Indeed, $\theta(\infty) = \infty$. Moreover $\theta(x) = \infty$ iff $x=\infty$ or $0$. The point $0$ is the only vertex belonging to the first level of the tree rooted at $\infty$. Moreover, $\psi(0) = -1$, which has order $2$ in $\F_{3^n}^*$.

If $\alpha \in \F_{3^n} \backslash \{-1, 1 \}$ is a $\theta$-periodic element, then $\psi(\alpha) \in \F_{3^n}^*$ and finding all the elements $\beta$ such that $\theta(\beta) = \alpha$ amounts to finding all the elements $\beta$ such that $\psi \circ s \circ \psi (\beta) = \alpha$. This latter is equivalent to $s(\psi(\beta)) = \psi(\alpha)$, namely $\psi(\beta)^2 = \psi(\alpha)^{-1}$. According to Lemma \ref{l1} the order of $\psi(\alpha)$, and consequently $\psi(\alpha)^{-1}$, is odd. Hence, $\left(\psi(\alpha)^{-1} \right)^{(3^n-1)/2} = 1$ in $\F_{3^n}$. Therefore, $\psi(\alpha)^{-1}$ is a quadratic residue in $\F_{3^n}^*$ and there are two distinct roots $r_1, r_2=-r_1$ of $x^2-\psi(\alpha)^{-1}$ in $\F_{3^n}$.  Being the map $\psi$ a bijection on $\Pro (\F_{3^n})$, it follows that $r_1 = \psi(\beta_1)$ and $r_2 = \psi (\beta_2)$ for two distinct elements $\beta_1$ and $\beta_2$ in $\F_{3^n}$. Moreover, since $\alpha$ is $\theta$-periodic, one among $\beta_1$ and $\beta_2$, let us say $\beta_1$, is $\theta$-periodic too and consequently $r_1$ has odd order. On the contrary $\beta_2$ is not $\theta$-periodic and $\ord (\psi(\beta_2)) = 2 \cdot \ord(\psi(\beta_1))$, namely the highest  of $2$ which divides $\ord (\psi(\beta_2))$ is $2$. 

The remaining statements regarding the levels $k>0$ will be proved by induction on $k$. Let us first consider the level $k=1$. If $e=1$, then there are no elements at the second level of the tree by Lemma \ref{l4}. In the case $e > 1$ consider the only element $\gamma$ belonging to the first level of the tree rooted at some $\theta$-periodic point of $\Pro (\F_{3^n}) \backslash \{-1, 1 \}$. We have proved that $2$ is the greatest power of $2$ which divides $\ord ({\psi(\gamma)})$. In virtue of Lemma \ref{l4} there are exactly two elements belonging to the level $2$ of the tree, whose image under the action of the map $\theta$ is $\gamma$.

Now we proceed with the inductive step. Suppose that for some integer $k > 1$ such that $k-1<e$ there are $2^{k-2}$ elements at the level $k-1$ of the tree and that each of these elements has two children. Moreover, if $\gamma$ is one of the elements at the level $k-1$, then $2^{k-1}$ is the greatest power of $2$ which divides $\ord (\psi(\gamma))$. Let $\beta$ any of the children of $\gamma$. Since $\theta(\beta) = \gamma$, we have that $\psi(\beta)^{-2} = \psi(\gamma)$. Then $2^k$ is the highest power of $2$ which divides $\ord(\psi(\beta))$. Finally, if $k<e$, then $\beta$ has two children, while, if $k=e$, then $\beta$ has no child by Lemma \ref{l4}.
\end{proof}

\subsection{An example: the graph associated with $\theta$ over the field $\F_{3^3}$} 
The field with $27$ elements can be constructed as the splitting field over $\F_3$ of the Conway polynomial $x^3-x+1$. In particular, if $\alpha$ denotes a root of such a polynomial, $\Pro(\F_{3^3}) = \left\{ \alpha^i: 0 \leq i \leq 25 \right\}  \cup \{ 0 \}  \cup \left\{ \infty \right\}$.

Below are represented the $3$ connected components of the graph. The labels of the vertices are the exponents of the powers $\alpha^i$, for $0 \leq i \leq 25$, the zero element (denoted by `0') and the point $\infty$. 

We notice that, according with Theorem \ref{l2}, the elements $\alpha^0 = 1$ and $\alpha^{13} = -1$ form a cycle of length $2$, while the cycle formed by $\infty$ has length $1$. Moreover, the set of odd integer divisors of $3^3-1$ greater than $1$ is $D = \{ 13 \}$. Since $\ord_{13} (-2) = 12$, then there is $\dfrac{1}{13} \cdot \phi(13) = 1$ cycle of length $12$. 

Finally, in accordance with Theorem \ref{l3}, any element belonging to a cycle is root of a binary tree of depth $1$.

\begin{center}
    \unitlength=3.7pt
    \begin{picture}(70, 70)(-20,-35)
    \gasset{Nw=3.6,Nh=3.6,Nmr=1.8,curvedepth=-0.5}
    \thinlines
    \footnotesize
    \node(N1)(20,0){$1$}
    \node(N2)(17,10){$20$}
    \node(N3)(10,17){$22$}
    \node(N4)(0,20){$11$}
    \node(N5)(-10,17){$3$}
    \node(N6)(-17,10){$8$}
    \node(N7)(-20,0){$14$}
    \node(N8)(-17,-10){$7$}
    \node(N9)(-10,-17){$9$}
    \node(N10)(0,-20){$24$}
    \node(N11)(10,-17){$16$}
    \node(N12)(17,-10){$21$}
    
    \node(N21)(30,0){$5$}
    \node(N22)(25,15){$25$}
    \node(N23)(15,25){$6$}
    \node(N24)(0,30){$4$}
    \node(N25)(-15,25){$15$}
    \node(N26)(-25,15){$23$}
    \node(N27)(-30,0){$18$}
    \node(N28)(-25,-15){$12$}
    \node(N29)(-15,-25){$19$}
    \node(N30)(0,-30){$17$}
    \node(N31)(15,-25){$2$}
    \node(N32)(25, -15){$10$}

    \drawedge(N1,N2){}
    \drawedge(N2,N3){}
    \drawedge(N3,N4){}
    \drawedge(N4,N5){}
    \drawedge(N5,N6){}
    \drawedge(N6,N7){}
    \drawedge(N7,N8){}
    \drawedge(N8,N9){}
    \drawedge(N9,N10){}
    \drawedge(N10,N11){}
    \drawedge(N11,N12){}
    \drawedge(N12,N1){}
    
    \gasset{curvedepth=0}
     
    \drawedge(N21,N1){}
    \drawedge(N22,N2){}
    \drawedge(N23,N3){}
    \drawedge(N24,N4){}
    \drawedge(N25,N5){}
    \drawedge(N26,N6){}
    \drawedge(N27,N7){}
    \drawedge(N28,N8){}
    \drawedge(N29,N9){}
    \drawedge(N30,N10){}
    \drawedge(N31,N11){}
    \drawedge(N32,N12){}
    
    \node(N40)(40,0){$13$}
    \node(N41)(50,0){$0$}
    \node(N50)(60,0){$\infty$}
    \node(N51)(60,10){`0'}

    \drawedge(N51,N50){}
    
     \drawloop[loopangle=-90](N50){}
     
     \gasset{curvedepth=1}
     \drawedge(N40,N41){}
    \drawedge(N41,N40){}

\end{picture}
\end{center}
\bibliography{Refs}
\end{document}